\documentclass
{amsart}
\usepackage{amssymb,amsmath,amsfonts,amsthm,amscd,enumerate}
\usepackage[inline]{enumitem}

\usepackage{graphics}
\usepackage{mathrsfs}
\usepackage{color}

\input xy
\xyoption{all}

\newcommand{\CC}{\mathbb C}

\theoremstyle{definition}
\newtheorem{definition}{Definition}[section]

\theoremstyle{plain}
\newtheorem{theo}[definition]{Theorem}
\newtheorem{prop}[definition]{Proposition}
\newtheorem{lem}[definition]{Lemma}
\newtheorem{cor}[definition]{Corollary}

\theoremstyle{remark}
\newtheorem{remark}[definition]{Remark}
\newtheorem{eg}[definition]{Example}

\newcommand{\R}{\mathbb R}
\newcommand{\C}{\mathbb C}
\newcommand{\Rtil}{\widetilde \R}
\newcommand{\Ctil}{\widetilde \C}
\newcommand{\N}{\mathbb N}

\newcommand{\eps}{\varepsilon}

\newcommand{\D}{\mathcal{D}}
\newcommand{\cinfty}{{\mathcal C}^\infty}
\newcommand{\vphi}{\varphi}
\newcommand{\comp}{\Subset}
\newcommand{\gs}{{\mathcal G}}
\newcommand{\ns}{{\mathcal N}}
\newcommand{\esm}{{\mathcal E}_M}

\newcommand{\Om}{\Omega}
\newcommand{\otilc}{\widetilde \Omega_c}
\newcommand{\obull}{\Omega^\bullet}

\newcommand{\st}[1]{{#1}^\circ}
\newcommand{\FR}{{^\bullet\R}}
\newcommand{\ext}[1]{{}^\bullet #1}
\newcommand{\gabs}[1]{\left| #1 \right|_\text{g}}
\newcommand{\Co}{{\mathcal C}^0}
\newcommand{\thick}[1]{\text{\rm th\!}\left(#1\right)}

\begin{document}
\title[New topologies on CGN and Fermat-Reyes theorem]{New topologies on Colombeau generalized numbers and the 
Fermat-Reyes theorem}

\author[P. Giordano]{Paolo Giordano}
\address{University of Vienna, Austria }
\email{paolo.giordano@univie.ac.at}
\thanks{P. Giordano was supported by Lise Meitner grant M1247-N13 of the Austrian Science Fund FWF}

\author[M. Kunzinger]{Michael Kunzinger}
\address{University of Vienna, Austria }
\email{michael.kunzinger@univie.ac.at}
\thanks{M. Kunzinger was supported by FWF grants P20525, P23714 and Y237-N13 of the Austrian Science Fund FWF}

\begin{abstract}
Based on the theory of Fermat reals we introduce new topologies  on spaces of Colombeau generalized points
and derive some of their fundamental properties. In particular, we obtain metric topologies on the space 
of near-standard generalized points that induce the standard Euclidean topology on the reals. We also
give a new description of the sharp topology in terms of the natural extension of the absolute value
(or of the defining semi-norms in the case of locally convex spaces) that allows to preserve a number
of classical notions. Building on a new point value characterization of Colombeau generalized functions
we prove a Fermat-Reyes theorem that forms the basis of an approach to differentiation on spaces
of generalized functions close to the classical one.


\end{abstract}

\maketitle

\section{Introduction}

In recent years, the structural theory of Colombeau generalized functions (see, e.g., \cite{C1,C2,Biag,MObook,GKOS})
has made important advances, especially concerning its topological and algebraic aspects 
(e.g., \cite{AJ,AJOS,DHPV,Gtop, Gtop2,OPS,Ob-Ve,S0,S,Ver10,Ver11}). In particular, the use of sharp topologies
has allowed to capture the inherent asymptotic properties of algebras of generalized functions 
in terms of suitable pseudo-metrics. In the present paper we continue this line of research in two main
directions. On the one hand, we study topologies on spaces of Colombeau generalized numbers. The guiding
principle here is to exploit the structural similarity of Colombeau's theory with the recently 
developed theory of Fermat reals (\cite{Gio09,Gio10a,Gio10b,Gio10c,Gio10e,GK}). This allows us to introduce new topologies,
the so-called Fermat- and $\omega$-topology, 
that are well-adapted to near-standard generalized points (a subset of generalized points sufficiently large
to characterize generalized functions by their point values). Moreover, we obtain a new description of the
sharp topology in terms of the natural extension of the absolute value (or, more generally, of the defining semi-norms
in the case of locally convex spaces) to generalized points. This approach allows to work with sharp topologies
in close analogy to classical (Euclidean or locally convex) topologies, preserving a number of fundamental tools.

Our second goal is to derive a Fermat-Reyes theorem in the context of Colombeau generalized functions. As in the
above topological considerations, this approach to differentiation is close to the classical theory and in our
view has considerable potential for further extensions, in particular in the direction of infinite-dimensional
analysis involving generalized functions. Along the way, we derive new point value characterizations of 
Colombeau generalized functions.

In the remainder of this introduction, we fix some notations from Colombeau's theory and give a brief introduction
to the theory of Fermat reals.

Let $\Om$ be an open subset of $\R^n$ and set $I:=(0,1]$. Then the (special) Colombeau algebra
on $\Om$ is defined as the quotient $\gs(\Om):=\esm(\Om)/\ns(\Om)$, where
\begin{multline*}
\esm(\Om) := \{ (u_\eps) \in \Co(I,\cinfty(\Omega)) \mid \forall K \comp\, \Om \ \forall \alpha\in \N_0^n\ \exists N\in \N : \\
\sup_{x\in K} |\partial^\alpha u_\eps(x)| = O(\eps^{-N})\}
\end{multline*}
\begin{multline*}
\ns(\Om) := \{ (u_\eps) \in \Co(I,\cinfty(\Omega)) \mid \forall K \comp \Om \ \forall \alpha\in \N_0^n\ \forall m\in \N : \\
\sup_{x\in K} |\partial^\alpha u_\eps(x)| = O(\eps^{m})\}
\end{multline*}
Here, and in what follows, every asymptotic relation is for $\eps \to 0^+$. We write $(u_\eps)$ for nets in $\esm(\Omega)$ and $u=[u_\eps]$ for the equivalence class in $\gs(\Om)$ 
of $(u_\eps)$. The mapping $\Omega\mapsto \gs(\Omega)$ is a fine and supple sheaf of differential algebras
and there exist sheaf morphisms embedding the space $\D'$ of Schwartz distributions into $\gs$ (cf.\ \cite{GKOS,PZ}).
The ring of constants in $\gs$, i.e., the subalgebra consisting of elements with 
representatives constant in $x$, is denoted by $\Rtil$ (or $\Ctil$  in the complex-valued setting).
In general we have $\widetilde{\Omega} = \Omega_M/\sim$, where $\Omega_M = \{ (x_\eps) \in \Co(I,\Omega)\mid \exists N\in \N: |x_\eps| = O(\eps^{-N})\}$ and $ (x_\eps) \sim (y_\eps) $ if $ |x_\eps - y_\eps| = O(\eps^m)$ for every $ m\in\N $.
Note that we always assume the representatives to depend continuously on $\eps$. This guarantees better algebraic
properties (e.g., the classical polynomials only have classical roots in $\Ctil$). Also, it does not make a difference
in the quotient whether one assumes smooth or merely continuous dependence on $\eps$. For an extended analysis,
see \cite{BK}.

The space of compactly supported generalized points $\otilc$ is defined as the quotient $\widetilde \Omega_c / \sim$,
where $\Omega_c := \{ (x_\eps) \in \Co(I,\Omega) \mid \exists K\comp \Omega\,\exists\eps_0\,\forall \eps<\eps_0: x_\eps\in K\}$ and $(x_\eps) \sim (y_\eps)$ if $|x_\eps - y_\eps| \to 0$ faster than any power of $\eps$.
Clearly $\Omega$ is embedded into $\otilc$ via $x\mapsto (\eps\mapsto x)$. Any Colombeau
generalized function $u\in \gs(\Om)$ naturally acts on $\otilc$ as $u(x):=[u_\eps(x_\eps)]$, and is in fact uniquely determined by
its values (in $\Rtil$) on compactly supported points (\cite{GKOS,Ndipl}), but not on standard points (elements of $\Om$).
This property is in marked contrast to the case of Schwartz distributions where no general point value
characterization is available and it opens the door to the direct transfer of many concepts from classical
(smooth) analysis to algebras of generalized functions. 

A refined version of this point value characterization was introduced in \cite{KK}: Borrowing from Nonstandard
Analysis, we call $x\in \otilc$ {\em near-standard} if there exists $\lim_{\eps \to 0^+} x_\eps =: \st{x}\in \Omega$ for one (hence any) representative of $x$. We write $x\approx \st{x}$. 
It can be shown that elements of $\gs(\Om)$ are
in fact uniquely determined by their point values on near-standard generalized points. For any near-standard $x\in \otilc$ we 
call $\delta(x):=x-\st{x}$ its infinitesimal part. Then $\delta(x) \approx 0$. The set of near-standard 
generalized points in $\otilc$ is denoted by $\obull$.

The possibility to define new topologies and the Fermat-Reyes theorem in the framework of Colombeau generalized functions derives from the theory of Fermat reals. For the sake of completeness, we recall here the definition of the ring $ \FR $ of Fermat reals, underscoring the possibilities of a transfer of ideas between the two settings.

Since $ \FR $ is defined, like $ \Rtil $, as a quotient set, we firstly need the following class of representatives.
\begin{definition}
\label{def:LittleOhPolynomials} 
We say that $x$ is a little-oh polynomial, and we write $x\in\R_{o}[t]$ if
\begin{enumerate}
\item $x:\R_{\ge0}\longrightarrow\R$ is continuous as $ t\to 0^+ $.
\item We can write $ x_t=r+\sum_{i=1}^{k}\alpha_{i}\cdot t^{a_{i}}+o(t) $ as $ t\to0^{+} $ for suitable $ k\in\N $, $ r,\alpha_{1},\dots,\alpha_{k}\in\R $, $ a_{1},\dots,a_{k}\in\R_{\ge0} $.
\end{enumerate}
\end{definition}
Hence, a little-oh polynomial $x\in\R_{o}[t]$ is a polynomial function with real coefficients, in the real variable $t\ge0$, with
generic positive powers of $t$, and up to a little-oh function as $t\to 0^{+}$. Therefore, as in $ \Rtil $ we do not have arbitrary representatives but more regular ones.
We can now define:
\begin{definition}
\label{def:equalityInFermatReals}
Let $x$, $y\in\R_{o}[t]$, then we say that $x\sim y$ if $x_t=y_t+o(t)$ as $t\to0^{+}$. Because it is easy to prove that $\sim$ is a ring congruence relation, we can define the quotient ring $\FR:=\R_{o}[t]/\sim$. Analogously, we can consider the set $ A_{o}[t] $ of little-oh polynomials with values in the open set $ A \subseteq \R $ and define $ \ext{A} := A_o[t]/\sim $.
\end{definition}
The ring $ \FR $ can be totally ordered with $ x\le y $ if there exist representatives $ x=[x_t]_\sim $ and $ y=[y_t]_\sim $ such that $ x_t \le y_t $ for $ t $ sufficiently small.

Every smooth function $ f:\R \longrightarrow \R $ can be extended to $ \FR $ as follows.
\begin{definition}
\label{def:extensionOfFunctions}
Let $A$ be an open subset of $\R$, and $f:A\longrightarrow\R$ a smooth function. For $[x]_\sim\in\ext A$ define $ \ext{f}([x]_{\sim}):=[f\circ x]_{\sim}$.
\end{definition}
In this way, e.g., every smooth function becomes equal to its tangent line in a first order infinitesimal neighbourhood
\[
\forall h \in D:\ \ext{f}(x+h) = \ext{f}(x) + h \cdot f'(x),
\]
where $ D:=\{h\in \FR \,|\, h^2=0 \} $ is the set of first order infinitesimals and $ x\in A $.
We refer to \cite{Gio10a} for further details. To see applications of this type of nilpotent infinitesimals to elementary physics,
see \cite{Gio10b}. For meaningful metrics on $\FR$ and for roots of (nilpotent) infinitesimals and their applications to fractional derivatives, see \cite{GK}.

Not every function we are interested in is of the form $ \ext{f} $. For example, $ f(x)=x+p $, for $ p:=[t^{1/2}]_\sim\in\FR  $ is not of this form because it takes standard reals into nonstandard values. It was to develop the calculus of this type of "quasi-standard" smooth functions that we proved the Fermat-Reyes theorem, see \cite{Gio10e}. Roughly speaking, a quasi-standard smooth function is a map that locally can be written as $ f(x)=\ext{\alpha}(p,x) $, where $ p $ is a parameter, possibly nonstandard, like in the previous example. Essentially, the Fermat-Reyes theorem states the existence and uniqueness of a quasi-standard smooth incremental ratio for every quasi-standard smooth function $ f $; see below for an analogous statement in the context of Colombeau generalized functions.

We finally underscore that the theory of Fermat reals can be developed very far: every smooth manifold can be analogously extended using this type of infinitesimals; more generally, this extension is applicable to every diffeological space (\cite{Igl}) obtaining a functor with very good preservation properties; the category of diffeological spaces is cartesian closed (\cite{Igl,Gio10c}) and embeds the category of smooth manifolds, so that these Fermat extensions can also be applied to infinite dimensional function spaces. For more details, see \cite{Gio10e,Gio09}.

\section{Topologies on Colombeau generalized points} \label{topsec}
The introduction of suitable topologies adapted to the asymptotic nature of the theory, the so-called sharp topologies
on spaces of Colombeau generalized functions,
is due to D.\ Scarpalezos in the early 1990ies (published only later in \cite{S0,S}, cf.\ also \cite{Biag}).
In \cite{Gtop,Gtop2}, this theory was then extended to a full-scale locally convex theory for algebras of
generalized functions. Sharp topologies have had a strong impact on the structural theory of Colombeau generalized
functions, cf., e.g., \cite{AJ, AJOS,M,GV}.

We briefly recall the definition and main properties of the sharp topology on $\Rtil$. 
Let 
\begin{align*}
v&:\esm \longrightarrow  (-\infty,\infty] \\
v&((u_\eps)) := \sup \{b\in \R \mid |u_\eps| = O(\eps^b) \}.
\end{align*}
Then $v((u_\eps)+(n_\eps)) = v((u_\eps))$ for all $(n_\eps) \in \ns$, so $v$ induces a well-defined pseudovaluation (again denoted by $v$) on $\Rtil$, and we have $v(u)=\infty$ if and only if $u=0$, $v(u\cdot w) \ge v(u) + v(w)$, and $v(u + w) \ge \min[v(u),v(w)]$, $v(u - w) = v(w - u)$. Letting $|-|_e: \Rtil \to [0,\infty)$, $|u|_e := \exp(-v(u))$ we conclude that $|u+v|_e \le \max(|u|_e, |v|_e)$, as well as $|uv|_e \le |u|_e|v|_e$.
Finally, we  obtain the translation invariant ultrametric
\begin{align*}
d&_s: \Rtil \times \Rtil \longrightarrow \R_+\\
d&_s(u,v) := |u-v|_e
\end{align*}
on $\Rtil$. $d_s$ generates the sharp topology on $\Rtil$. Analogously one defines the sharp topology on $\Rtil^n \cong
\widetilde{\R^n}$, which coincides with the product topology. $d_s$ is a translation-invariant ultrametric and $\Rtil$ 
with the sharp topology is complete. Usually, $\otilc \subseteq \Rtil^n$ is endowed with the trace of the sharp topology. 

More generally, as C.\ Garetto has shown in \cite{Gtop,Gtop2}
one may functorially assign a space of Colombeau generalized functions $\gs_E$ to any
given locally convex space $E$: Suppose that the topology of $E$ is induced by a family of
seminorms $(p_\alpha)_{\alpha\in A}$ and set
\begin{eqnarray*}
\mathcal{M}_E &:=& \{ (u_\eps) \in \cinfty(I,E) \mid \forall \alpha\, \exists N\,: p_\alpha(u_\eps)= O(\eps^{-N})\}\\
\mathcal{N}_E &:=& \{ (u_\eps) \in \cinfty(I,E) \mid \forall \alpha\, \forall m\,: p_\alpha(u_\eps)= O(\eps^{m})\}\\
\gs_E &:=& \mathcal{M}_E/\mathcal{N}_E
\end{eqnarray*}
Then $\gs_E$ is a $\widetilde{\C}$-module. The special Colombeau algebra as defined above
corresponds to the special case $E=\cinfty(\Omega)$.

On $\gs_E$ we introduce the family of valuations 
$$
v_\alpha(u):= \sup\{b\in \R \mid p_\alpha(u_\eps)=O(\eps^b)\}.
$$
These valuations induce ultra-pseudo-seminorms by
$$
\mathcal{P}_\alpha := e^{-v_\alpha},
$$
and the family $(\mathcal{P}_\alpha)_{\alpha\in A}$ defines the so-called sharp topology on $\gs_E$.

The sharp topology is Hausdorff and induces the discrete topology on $\R\subseteq \Rtil$. Indeed, for $ a,b \in \R$, $a\not= b$, 
$$
d_s(a,b) = |a-b|_e = \exp(-v(a-b)) = \exp(0) = 1\,.
$$
This is a necessary and general result. Indeed, let us suppose that $ (R,\tau) $ is a topological space containing the reals $ \R \subseteq R $. Moreover, let us also assume that $R$ is endowed with 
a partially ordered ring structure $ (R,+,\cdot,\le) $ extending the usual structure on the reals. We can hence define that $ h\in R $ is an infinitesimal, and we will write $ h\approx 0 $, if
\begin{equation}
\label{inf}
\forall r\in \R_{>0}:\ -r < h < r.
\end{equation}
Finally, we say that $ x, y \in R $ are infinitely close if $ y-x \approx 0 $, and we define the \emph{monad} of $ x\in R $ as 
\[
\mu(x):=\{y\in R \,|\, y \approx x \}.
\]
Then we have the following
\begin{prop}
\label{infNeigh}
Under the above assumptions, if $ (R,\tau) $ contains infinitesimal neighbourhoods of the reals, i.e.,
\[
\forall r\in \R\ \exists U \in \tau:\ r\in U \quad \text{and} \quad U\subseteq \mu(r),
\]
then the topology induced by $ (R,\tau) $ on $ \R $ is discrete.
\end{prop}
\begin{proof}
Let $ r\in \R $ and let $ U \subseteq \mu(r) $ be an infinitesimal neighbourhood of $ r $, then $ U\cap \R = \{ r \} $ by the definition of an infinitesimal.
\end{proof}
Of course, $ \Rtil $ with the sharp topology verifies all the assumptions of this proposition.

Also, since $|\ |_e$ measures the $\eps$-asymptotics of elements of $\Rtil$, it does not distinguish between
scalar multiples: $|\lambda u|_e = |u|_e$ for all $\lambda\in\R\setminus\{0\}$. Such properties may 
appear unwanted when comparing with analysis over the reals. However,
we will see below how to arrive at the sharp topology using the absolute value in $ \Rtil $, that instead verifies $ |\lambda \cdot x| = |\lambda| \cdot |x| $ and many other properties very similar to those of classical analyis.

Proposition \ref{infNeigh} implies that every metric on $ \Rtil $ that admits infinitesimal neighbourhoods cannot extend the usual Euclidean metric. Stated differently, every metric on $ \Rtil $ extending the standard metric on $ \R $ cannot have purely infinitesimal neighbourhoods. To show that such a metric exists, 
we transfer ideas from the theory of Fermat reals (see \cite{GK}) to spaces of Colombeau generalized points. 
\begin{lem} \label{complem}
Let $D:=\{x\in \Rtil^n \mid x\approx 0\}$ be the set of infinitesimals of $ \Rtil^n $, then
 $D$ is a closed (hence complete) subset of $(\Rtil^n,d_s)$.
\end{lem}
\begin{proof} Let $(x_m)$ be a sequence in $D$ converging to some $x\in \Rtil^n$. Then $v(x_m-x)\to \infty$
as $m\to \infty$. Choose $m_0$ such that $v(x_m-x) > 1$ for all $m\ge m_0$. Then there exists some $\eps_0>0$
with $|x_{m_0,\eps} - x_\eps| \le \eps$ for $\eps<\eps_0$. Since $x_{m_0}\in D$ this implies that $x_\eps \to 0$
for $\eps\to 0$, i.e., $x\in D$.
\end{proof}
For any open subset $\Omega$ of $\R^n$ we next define two new distance functions $d_F$ and $d_\omega$ on $\obull$
as follows: for $x$, $y \in \obull$ let
$$
d_F(x,y) := |\st{x} - \st{y}|, \qquad d_\omega(x,y) := |\st{x} - \st{y}| + d_s(\delta(x),\delta(y)).
$$ 
The corresponding topologies are called the {\em Fermat topology} and the {\em $\omega$-topology} on $\obull$, respectively.
\begin{prop}\label{F-omega-sharp-top} \
\begin{itemize}
\item[(i)] $d_F$ is a pseudometric on $\obull$.
\item[(ii)] $d_\omega$ is a metric on $\obull$.
\item[(iii)] The $\omega$-topology is strictly finer than the Fermat topology.
\item[(iv)] The sharp topology on $\obull$ is strictly finer than the $\omega$-topology.
\end{itemize}
\end{prop}
\begin{proof} (i) and (ii) are evident, as is the fact that the $\omega$-topology is finer than the Fermat topology.
To see that it is strictly finer let $0<R<1$ and let $x\in \obull$. We show that, for each $r>0$, $B_r(x,d_F)\not\subseteq 
B_R(x,d_\omega)$. Choose $s>0$ such that $e^{-s} > R$ and let $y:=[x_\eps+\eps^s]$. Then $\st{y} = \st{x}$, so $y\in
B_r(x,d_F)$ for each $r>0$. However,
$$
d_\omega(x,y) = d_s(\delta(x),\delta(y)) = |\eps^s|_e = e^{-s} > R.
$$
\noindent(iv) We first note that for $r<1$ we have $\{x\in \obull\mid d_s(x,0)<r\} \subseteq
\{x\in \obull\mid d_\omega(x,0)<r\}$. Indeed, $d_s(x,0)<1$ implies $\st{x} = 0$, so $d_\omega(x,0) =
d_s(\delta(x),0)=d_s(x,0)$ for such $x$. By translation-invariance this shows that the sharp 
topology is finer than the $\omega$-topology on $\obull$. That it is strictly finer follows from
Example \ref{conv_ex} below. 
\end{proof}
\begin{remark}\label{toprem}
The following observations are immediate from the definitions:
\begin{enumerate}[labelindent=\parindent,leftmargin=*,label=(\roman*),align=left ]
\item If $x$, $y\in \Omega$ then $d_F(x,y) = |x-y| = d_\omega(x,y)$, hence both $d_F$ and $d_\omega$ induce
the usual Euclidean topology on $\Omega$.
\item By Proposition \ref{infNeigh}, both the Fermat and the $ \omega $-topology don't contain infinitesimal neighbourhoods. E.g., $ d_\omega(x,y) < r$ iff there are $ b,c\in\R_{>0} $ such that $ b+c < r $, $ |\st{x} - \st{y}| \le b$ and 
$|\delta(x) - \delta(y)|_e \le c$ and therefore every neighbourhood of $ x $ always contains both standard and infinitely close points.
\item\label{isomSpaces} The map $x\mapsto (\st{x},\delta(x))$, $\obull \to \Omega\times (D,d_s)$ is an isomorphism of metric spaces.
\item Let $x\in \Omega$ and $y$, $z\in \mu(x)$. Then $d_\omega(y,z) = e^{-v(\delta(y),\delta(z))}$. 
\item The Fermat topology is not Hausdorff.
\end{enumerate}
\end{remark}
\begin{prop} $(\R^n)^\bullet$ is complete with respect to the Fermat topology.
\end{prop}
\begin{proof} Let $(x_m)$ be a Cauchy sequence w.r.t.\ $d_F$ in $(\R^n)^\bullet$. Then $(\st{x}_m)$ is
a Cauchy sequence in $\R^n$ hence converges to some $x\in\R^n$. Therefore, $(x_m)$ converges to any
element of $\mu(x)$.
\end{proof}
By Rem.\ \ref{toprem} \ref{isomSpaces} and Lemma \ref{complem} we also have:
\begin{prop} $(\R^n)^\bullet$ is complete with respect to the $\omega$-topology.
\end{prop} 
This result also shows that Cauchy completeness with respect to a metric extending the classical Euclidean metric is not incompatible with the existence of infinitesimals. Of course, $ \Rtil $ is not Dedekind complete because the set of infinitesimals $ D $ is bounded but without least upper bound.

In order to compare with the $\omega$-topology, we recall the following continuity result for the sharp topology:
\begin{lem}
\label{contlem} 
Let $u\in \gs(\Omega)$. Then the induced map $u: \otilc \to \Rtil$ is continuous with respect to the sharp topology.
\end{lem}
\begin{proof} Let $x\in \otilc$ and let $(x_k)$ be a sequence in $\otilc$ converging to $x$ with respect to 
$d_s$. Let $k_0$ be such that $v(x-x_k)\ge 1$ for all $k\ge k_0$. Then for any such $k$ and $\eps$ sufficiently
small the line connecting $x_{k,\eps}$ and $x_\eps$ is contained in $\Omega$ and we have 
$$
u_\eps(x_{k,\eps})-u_\eps(x_\eps) = \int_0^1 Du_\eps(x_\eps + \tau(x_{k,\eps}-x_\eps))\,d\tau \cdot (x_{k,\eps}-x_\eps), 
$$
so $|u_\eps(x_{k,\eps})-u_\eps(x_\eps)| \le \eps^{-N}|x_{k,\eps}-x_\eps|$ for some $N>0$ and $\eps$ small by moderateness 
of $(u_\eps)$. Thus, finally, $|u(x_k) - u(x)|_e \le e^N |x_k-x|_e \to 0 \quad (k\to \infty)$.
\end{proof}
However, if $u\in \gs(\Omega)$, the associated map $u: \obull \to \Rtil$ is not continuous in general, as is demonstrated
by the following example.
\begin{eg} \label{conv_ex} 
Let $\Omega=\R$, $u_\eps(x)=\frac{x}{\eps}$ and $x_k:=\frac{1}{k}$. Then $u\in \gs(\Omega)$, $d_\omega(x_k,0)\to 0$,
but
$$
d_s(u(x_k),0) = \left|\left[\frac{1}{k\eps}\right]\right|_e = \left|\left[\frac{1}{\eps}\right]\right|_e = e^1 \not\to 0.
$$
\end{eg}
In a certain sense, this is another necessary and general result. Indeed, suppose that $(M,d)$ is a metric space containing the real interval $[0,1]$ and extending the usual metric on these reals
\begin{equation}
\label{extensionDistReals}
\forall x, y\in [0,1]:\ d(x,y) = |x - y|.
\end{equation}
Then $ d\left( \frac{1}{k},0 \right) \to 0$. This gives an intuitive contradiction if in $ M $ we have at least one non zero (actual) infinitesimal $ h\in M\setminus\{0\} $. In fact, $ d\left(\frac{1}{k} , h \right) \to d(0,h)>0 $. Therefore, it seems that although the sequence $ \left(\frac{1}{k}\right)_{k\in\N} $ approaches $ 0 $ it stays always far away from every infinitesimal. Let us note that we didn't formally use the condition \eqref{inf} of being an infinitesimal, so the contradiction is only intuitive.

If there are infinitesimal neighbourhoods, it is intuitively clear that $ \frac{1}{k} \not \to 0$ because with only finite $ k\in\N $ one cannot eventually arrive in an infinitesimal neighbourhood of $ 0 $. We can also say that in the presence of actual infinitesimals, any tortoise will be a scholar of Zeno because to reach its goal it needs an infinite number of steps. As stated above, both the Fermat and the $ \omega $-metrics do not have this type of infinitesimals. In the sharp metric there are infinitesimal balls with finite radius which bounds the order of its infinitesimals. However, there is another possibility to have a metric extending the Euclidean one and admitting infinitesimal neighbourhoods: we must permit to measure the radius of a ball using also infinitesimals. This can be done very naturally considering the absolute value on $ \Rtil $.
\begin{definition}
Let $ x, y \in \Rtil$, then
\begin{align*}
\gabs{x-y} &:= [|x_\eps - y_\eps|] \in \Rtil \\
\min(x,y)  &:= [\min(x_\eps, y_\eps)] \in \Rtil \\
\max(x,y)  &:= [\max(x_\eps, y_\eps)] \in \Rtil.
\end{align*}
\end{definition}
\noindent Note that these are the continuous versions of the smooth ones considered in \cite{BK}, which correspond to each other by the natural isomorphism $\tau_{sm}$ from \cite{BK}, sec.\ 3. The following properties are immediate:
\begin{align*}
\label{propAbs}
&\gabs{x} = \max(x,-x) \\
&\gabs{x}\ge 0 \\
&\gabs{x}=0 \Rightarrow x=0 \\
&\gabs{\lambda \cdot x} =  \gabs{\lambda} \cdot \gabs{x}\\
&\gabs{x+y} \le \gabs{x} + \gabs{y} \\
&\gabs{r} = |r|\quad \forall r \in \R.
\end{align*}
Here, $ \le $ is the usual order relation on $ \Rtil $, i.e. $ x \le y $ iff there are representatives $ (x_\eps) $ and $ (y_\eps) $ such that $ x_\eps \le y_\eps $ for $ \eps $ sufficiently small. 

Usually, the order topology is defined for totally ordered sets whereas $ \Rtil $ is only partially ordered. 
For example, we only have
\begin{equation}
\label{lackIntersection}
(\max(a,c), \min(b,d)) \subseteq (a,b) \cap (c,d) \subseteq [\max(a,c), \min(b,d)]
\end{equation}
and for non comparable $ a, c $ we can have that $ \max(a,c) \in (a,b) \cap (c,d) $ (e.g. $ a_\eps=\sin\left(\frac{1}{\eps}\right) $ and $ c_\eps=\cos\left(\frac{1}{\eps}\right) $). It is interesting to note that a different notion of interval, i.e.
\begin{equation}
\label{interval}
(a,b) := \{ x\in \Rtil \mid a\le x \le b,\ a-x,\ b-x \text{ invertible}  \}
\end{equation}
on the contrary verifies the expected property $ (a,b) \cap (c,d) = (\max(a,c), \min(b,d)) $.
\\
For this reason, in our point of view, also the non-optimal property \eqref{lackIntersection} derives from 
the presence in $ \Rtil $ of zero divisors. 

We can indeed state that even if in $ \Rtil $ there are non invertible infinitesimals (see Theorem 1.2.39 in \cite{GKOS} for a characterization), only invertible infinitesimals like $ (c\eps^q) $ play a central role (see e.g.\ \cite{Gtop2}). The analogous theory of generalized functions developed using nonstandard analysis methods (see e.g.\ \cite{ToVe08}), where one fixes $ \rho \in {}^*\R_{>0} $, $ \rho \approx 0 $, and defines
\begin{align*}
\mathcal{M}_\rho &:= \{ x\in {}^*\R \mid \exists N \in \N:\ |x| \le \rho^{-N} \} \\
\mathcal{N}_\rho &:= \{ x\in {}^*\R \mid \forall n\in \N:\ |x| \le \rho^n \} \\
{}^\rho \R &:= \mathcal{M}_\rho / \mathcal{N}_\rho,
\end{align*}
is able to avoid this type of infinitesimals thanks to the use of an ultrafilter. In fact, if $ x\in \mathcal{M}_\rho \setminus \mathcal{N}_\rho $ then $ |x| \not \le \rho^n $ for some $ n\in \N $ and hence $ |x| \ge \rho^n $ because the order relation in $ {}^*\R $ is total (here the use of an ultrafilter is essential). The property $ |x| \ge \rho^n $ corresponds, intuitively, to the condition of being strictly nonzero in $ \Rtil $.
Recall that a generalized number $x\in\Rtil$ is invertible if and only if it is {\em strictly nonzero}, i.e., if and
only if there exists some $m$ such that $|x_\eps| > \eps^m$ for $\eps$ sufficiently small (\cite{GKOS}, 1.2.38). 

Avoiding zero divisors, the nonstandard theory is formally more powerful, since $ {}^\rho\C $ is algebraically closed and the Hahn-Banach theorem holds (\cite{ToVe08}), whereas in the standard version both these results fail (\cite{Ver10}). From this point of view, it is also important to recall that a version of the Hahn-Banach theorem holds for ultra-metric normed linear spaces {\em over a subfield} of $ \Rtil $ (\cite{M}).

Following this line of thought, we are more interested to consider on $ \Rtil $ the topology induced by $ \gabs{-} $ but using only balls with invertible radius (or using the notion of interval defined in \eqref{interval}), which is exactly the sharp topology, as stated in the following
\begin{theo}
\label{ThmSharpWithGabs}
A subset $ U\subseteq \Rtil $ is open in the sharp topology if and only if
\begin{equation}
\label{sharpWithGabs}
\forall x\in U\,\exists \rho\in \Rtil_{>0}:\ \rho \text{\rm\ is invertible and } B_\rho^\text{\rm g}(x)\subseteq U.
\end{equation}
\end{theo}
\begin{proof}
Consider a ball $ B^\text{s}_r(x) $ in the sharp topology, where $ r\in\R_{>0} $. Take any $ q>- \log (r/2) $ and set $ \rho_\eps := \eps^q $, then $ \rho $ is invertible and it is not hard to prove that $ B^\text{g}_\rho(x) \subseteq B^\text{s}_r(x) $. This shows that the sharp topology is finer than the topology defined by (\ref{sharpWithGabs}). Vice versa, if we take a ball $ B^\text{g}_\rho(x) $ in the order topology, but with $ \rho $ invertible, then $ \rho_\eps \ge \eps^q $ for some representative $ (\rho_\eps) $ of $ \rho $ and some $ q \in \N $. It suffices to consider any $ r \in \R_{>0} $ such that $ \log r < -q $ to have that $ B^\text{s}_r(x) \subseteq B^\text{g}_\rho(x) $.
\end{proof}
It follows that the sharp topology can equivalently be characterized using the absolute value $ \gabs{-} $ on $ \Rtil $, which extends the ordinary one on the reals and for which very good properties hold. 
\begin{eg}
\ 
\begin{enumerate}[labelindent=\parindent,leftmargin=*,label=(\roman*),align=left ]
\item Any $ l \in \Rtil $ is the limit of $ f:\Rtil \longrightarrow \Rtil $ as $ x\to x_0 $ in the sharp topology iff
\[
\forall \eta\in \Rtil_{>0}^*\, \exists \delta\in \Rtil_{>0}^*\, \forall x\in \Rtil:\, \gabs{x-x_0}<\delta \Rightarrow \gabs{f(x)-l}< \eta.
\]
where $ \Rtil^*:=\{x \in \Rtil \,|\, x\text{ is invertible}\} $. 
This includes the case $ \eta \approx 0 $, as stated informally above.
\item If $ r\in \R $ and $ \rho\in\Rtil_{>0} $ is invertible, the trace on $ \R $ of the ball $ B_\rho^{\text{g}}(r) $ w.r.t. $ \gabs{-} $ is always
\[
B_\rho^{\text{g}}(r) \cap \R = \{ s\in \R \,|\, |s-r|<\rho \},
\]
so
\begin{align*}
B_\rho^{\text{g}}(r)  \cap \R &=(r-\rho , r+\rho) \quad \text{if}\quad\rho\in\R_{>0} \\
B_\rho^{\text{g}}(r)  \cap \R &= \{ r \}  \quad \text{if}\quad\rho \approx 0.
\end{align*}
Hence, in accordance with Proposition \ref{infNeigh}, the induced topology on $ \R $ is discrete, whereas considering only standard balls we regain the classical Euclidean topology.
\item Every sphere $ S_r(x) := \{ y\in \Rtil^n \mid d_s(y,x) = r \} $ is of course closed, but also open. This is a general result for every ultrametric space, but in our context we have 
\begin{equation}
\label{sphere}
S_r(x) = \bigcup\{ B_\rho^\text{g}(y) \mid y\in S_r(x) \ , \ \exists q>-\log r:\ \rho = [(\eps^q)] \}.
\end{equation}
In fact, if $ z \in B_\rho^\text{g}(y) $ with $ y\in S_r(x) $, then $ \gabs{z - y} < \rho $ and by Krull sharpening $ d_s(x,z) = \max\{ d_s(x,y), d_s(y,z) \} = \max(r,e^{-q})=r$. Moreover, if $ r<1 $ every ball $ B_\rho^\text{g}(y) $ is infinitesimal and hence the sphere $ S_r(x) \subseteq \mu(x) $ is an infinitesimal neighbourhood of $ x $.
\item By the proof of Lemma \ref{contlem} it follows that every $ u \in \gs(\Omega) $ is locally Lipschitz w.r.t. the Fermat topology. More precisely, if $ x_0 \in \Omega^\bullet $ and we take $ r\in \R_{>0} $ such that the closure of the Fermat ball $ U:= \{ x \in \Omega^\bullet \mid d_F(x,x_0) \le r \} $ is contained in $ \Omega^\bullet $, then
\begin{equation}
\label{CGFLocLip}
\gabs{u(x) - u(y)} \le K \cdot \gabs{x-y}\qquad \forall x,y \in U
\end{equation}
where $ K \in \Rtil $. Note that from \eqref{CGFLocLip}, a standard Lipschitz condition w.r.t.\ the sharp topology follows, i.e. $ |u(x)-u(y)|_e \le k \cdot |x-y|_e $, with $ k\in \R $.
\end{enumerate}
\end{eg}
The strict relations between order and sharp topology imply that strange functions like
\begin{equation}
\label{contDisc}
i(x) := 
\begin{cases}
1 & \text{if $ x \approx 0 $} \\
0 & \text{otherwise}
\end{cases}
\end{equation}
are continuous. Once again this is a general result.
\begin{prop}
Under the assumptions of Proposition \ref{infNeigh}, let $ \delta \approx 0 $ be a positive infinitesimal, $ \delta \in R_{>0} $, such that
$$
\forall x, x_0\in R:\ x_0 - \delta < x < x_0 + \delta \ \Rightarrow\  \left( x_0 \approx 0 \iff x \approx 0\right). 
$$
Then
$$
\forall x, x_0 \in R:\ x_0 - \delta < x < x_0 + \delta \ \Rightarrow\ i(x)=i(x_0),
$$
where $ i $ is the function defined in (\ref{contDisc}).
\end{prop}
\noindent Therefore, the function $ i $ is continuous in the sharp topology, too.
This also shows that an intermediate value theorem for continuous functions in the sharp topology cannot hold. However, let us note that the function $ i $ is still locally Lipschitz, both in the sharp and in the Fermat topology (take $ K $ any infinite number in \eqref{CGFLocLip}).

Proceeding as above, we can analogously consider the spaces $ \Rtil^n $ and $ \gs(\Omega) $ defining
\begin{align*}
\Vert x \Vert_\text{g} &:= \left(\sum_{i=1}^n x_i^2\right)^{1/2}\in \Rtil \quad \text{if} \quad x\in \Rtil^n \\
\Vert u \Vert_{K\alpha}^\text{g} &:= \left[ \sup_{x\in K}\Vert \partial^\alpha u_\eps(x) \Vert \right] \in \Rtil
\quad \text{if}\quad u\in \gs(\Omega), K\Subset\Omega, \alpha \in \N_0^n.
\end{align*}
More generally, let $E$ be a locally convex vector space whose topology is induced by the family
of seminorms $(p_\alpha)_{\alpha\in A}$. Then the $p_\alpha$ uniquely extend to maps $p_\alpha^\text{g}: \gs_E\to \Rtil$
and the reasoning of the proof of Theorem \ref{sharpWithGabs} shows that the sharp topology on $\gs_E$ is
generated by the sets $\{u\in \gs_E \mid p_\alpha^\text{g}(u) < \rho\}$, where $\rho$ runs through all invertible
elements of $\Rtil_{>0}$ and $\alpha\in A$.

\section{The Fermat-Reyes theorem in $\gs(\Om)$}
As a prerequisite for our investigation of the Fermat-Reyes theorem we first demonstrate that
a further reduction of the set of generalized points characterizing any Colombeau generalized function is feasible.
\begin{theo} 
\label{invth}
Let $u\in \gs(\R)$. Then $u=0$ if and only if $u(x)=0$ for all $x\in \R^\bullet$ that are invertible.
\end{theo} 
\begin{proof} Since $u$ is entirely determined by its point values on elements of $\R^\bullet$ it suffices to show
that if $u(x)=0$ for all invertible $x\in \R^\bullet$ then the same follows for all $x\in \R^\bullet$. 
To see this we first note that $u(0)=0$: Let $x_k:=[\eps^k]$. Then each $x_k$ is invertible and $|x_k|_e=e^{-k}\to 0$
as $k\to \infty$. By Lemma \ref{contlem}, $0=u(x_k)\to u(0)$.

Suppose now that $u\not=0$. Then there exists some $x\in \Rtil_c$ with $u(x)\not=0$ and by the above $x\not=0$, but
$\st{x} = 0$ (otherwise $x$ would be invertible).
Since $u(x)\not=0$, there exists a representative $(u_\eps)$ of $u$, some $m_1$ and a sequence $1 >\  \eps_k\searrow 0$
such that $|u_{\eps_k}(x_{\eps_k})|\ge \eps_{k}^{m_1}$ for all $k$.  Since $\st{x}=0$, we can suppose to have chosen $(\eps_k)_k$ so that $0\le x_{\eps_{k+1}} < x_{\eps_k}$ or $x_{\eps_k} < x_{\eps_{k+1}} \le 0$. We will proceed in the first case, the second being analogous.

We will show that some subsequence of $x_{\eps_k}$
gives rise to an invertible generalized number by first establishing that 
\begin{equation}\label{star}
\exists m_2 \ \forall l \ \exists k_l\ge l: \ |x_{\eps_{k_l}}| \ge \eps_{k_l}^{m_2}
\end{equation}
Suppose to the contrary that $\forall m$ $\exists l_m$ $\forall k\ge l_m:$ $|x_{\eps_k}| < \eps_k^m$. We can assume that $l_{k+1} > l_k$. Construct inductively a smooth map $\eps \mapsto r_\eps$ as follows: To begin with, connect the points
$x_{\eps_{l_1}}< \cdots < x_{\eps_{l_0}}$ by some smooth curve $r$ ($r_{\eps_i} = x_{\eps_i}$ for $l_0\le i\le l_1$) in such a way that $|r_\eps|< \eps^0$ on $[\eps_{l_1}, \eps_{l_0}]$; note that $|x_{\eps_{l_0}}| < \eps_{l_0}^0$ and $|x_{\eps_{l_1}}| < \eps_{l_1}^1 < 1$.  In the next step, extend $r$ smoothly through the points $x_{\eps_{l_2}}< \cdots < x_{\eps_{l_1}}$ in such a way that $|r_\eps|<\eps^1$ on $[\eps_{l_2},\eps_{l_1}]$; note that $|x_{\eps_{l_2}}| < \eps_{l_2}^2 < \eps_{l_2}$. Iterating this procedure we obtain a smooth map $r:(0,1] \to \R$ such that  $r_{\eps_{l_m}} = x_{\eps_{l_m}}$ for all $m\in\N_0$ and $|r_\eps|<\eps^m$ for $\eps<\eps_{l_m}$. Thus $r$ is a representative of $0\in \Rtil_c$. However, 
$$
|u_{\eps_{l_m}}(r_{\eps_{l_m}})| = |u_{\eps_{l_m}}(x_{\eps_{l_m}})| \ge \eps_{l_m}^{m_1}
$$
for all $m\in \N_0$, so $u(0) = u(r) \not=0$, a contradiction.

This proves the existence of $m_2$ as claimed and by extracting a subsequence we may additionally assume 
that $k_l < k_{l+1}$ for all $l$ in (\ref{star}). 
Let $\vphi: \R\to [0,1]$ be a smooth map such that $\vphi(t) = 0$ for $t \ \le\  0$ and $\vphi(t)=1$ for $t\ge 1$.
Now define $y:(0,1] \to \R$ by
$$
y_\eps := \left\{ 
\begin{array}{ll}
x_{\eps_{k+1}} + \vphi\left(\frac{\eps-\eps_{k+1}}{\eps_k - \eps_{k+1}}\right) (x_{\eps_k} - x_{\eps_{k+1}}) & \eps_{k+1} \le \eps \le \eps_k\\
x_{\eps_1} & \eps_1\le \eps \le 1
\end{array}
\right.
$$
Then $y$ is a smooth parametrization of the polygon through the $x_{\eps_k}$ (cf.\ the special curve lemma in \cite {KM}, 2.8).
Since $|y_{\eps_k}| = |x_{\eps_k}| \ge \eps_k^{m_2}$ for all $k$ and since $\{(\eps,x)\in (0,1] \times \R \mid x\ge \eps^{m_2}\}$ is convex,
$|y_\eps|\ge \eps^{m_2}$ for all $\eps\in (0,1]$. Hence $(y_\eps)$ represents an invertible element $y$ of $\Rtil_c$. Moreover, 
for $\eps\le \eps_k$, $|y_\eps| \le |x_{\eps_k}| + 2|x_{\eps_{k+1}}|$, so $y\in\R^\bullet$ with $\st{y} = 0$.
But $u(y)\not=0$ since
$$
|u_{\eps_{k_l}}(y_{\eps_{k_l}})| = |u_{\eps_{k_l}}(x_{\eps_{k_l}})| \ge \eps_{k_l}^{m_1},
$$
a contradiction.
\end{proof}
\begin{cor} \label{invcor} Let $\Omega\subseteq \R^n$ be open and let $u\in \gs(\Omega)$. Then $u=0$ if and only if $u(x)=0$ for all $x\in \obull$ with $| x | $ invertible in $\Rtil$.
\end{cor}
\begin{proof} We first extend the validity of Th.\ \ref{invth} to open subsets $\Omega$ of $\R$. For $0\not\in \Omega$, every element of 
$\obull$ is invertible, so the result is obvious. Otherwise, the proof of Th.\ \ref{invth} applies. 

Now let $\Omega\subseteq \R^n$ ($n>1$) and suppose that $u(x)=0$ for all $x\in \obull$ with $ \Vert x \Vert $ invertible. Let $x\in \otilc$ arbitrary and consider the map $\tilde u :=y\mapsto [u_\eps(y,x_{2,\eps},\dots,x_{n,\eps})] \in \gs(\mathrm{pr}_1(\Omega))$. If $y\in \mathrm{pr}_1(\Omega)^\bullet$ is invertible, then so is $| [y_\eps,x_{2\eps},\dots,x_{n\eps}] |$, hence $\tilde u(y)=0$. It follows that $\tilde u = 0$ in $\gs(\mathrm{pr}_1(\Omega))$.
Therefore, $u(x)=0$ for all $x\in \otilc$, i.e., $u=0$.
\end{proof}
The Fermat-Reyes theorem states the existence and uniqueness of a generalized function acting as the incremental ratio of a given $ f\in \gs(U) $. The natural domain of definition of this incremental ratio is given by the following
\begin{definition} For $a$, $b\in \R^n$ denote by $\overrightarrow{[a,b]}$ the segment 
$\{a+s(b-a) \mid s\in \R \ , \ 0\le s \le 1\}$. For $U\subseteq \R$ open, the {\em thickening}
of $U$ is
$$
\thick{U} := \{(x,h)\in \R^{2n} \mid \overrightarrow{[x,x+h]} \subseteq U \}.
$$ 
\end{definition}
\noindent The thickening $ \thick{U} $ is an open subset of $ \R^{2n} $ (see \cite{AMR}).
\begin{lem} 
\label{nearStdThickIsOpen}
For any $U\subseteq \R^n$ open, the set 
$$
\thick{U}^\bullet = \{ (x,h) \in (\R^{2n})^\bullet \mid (\st{x},\st{h}) \in \thick{U} \}
$$
of near standard points of the thickening is open in the Fermat topology on $(\R^{2n})^\bullet$ (hence also in the $ \omega $- and in the sharp topology by Proposition \ref{F-omega-sharp-top}).
\end{lem}
\begin{proof}
Take $ (x,h)\in\thick{U}^\bullet $, then $ K:=\{\st{x}+s\cdot \st{h} \mid s\in\R \ , \ 0\le s \le 1\} $ is compact and contained in the open set $ U $. Let $ 2a:=d(K,\R^n\setminus U)>0 $ be the distance of $ K $ from the complement of $ U $, so that $ B_a(c)\subseteq U $ for every $ c\in K $. Let $(y,k)\in  (\R^{2n})^\bullet$ such that $d_F(x,y)<a/2$ as well as $d_F(h,k)<a/2$, then we claim that $ (y,k) \in \thick{U}^\bullet $; in fact, if $ s\in \R $, $ 0\le s \le 1 $, then
\begin{align*}
|\st{y}+s\st{k} - \st{x}-s\st{h}| &\le |\st{y}-\st{x}|+|s|\cdot |\st{k}-\st{h}| \\
          & < \frac{a}{2}+1\cdot \frac{a}{2}=a.
\end{align*}
Therefore, $ \st{y}+s\st{k} \in B_a(c) \subseteq U$ where $ c= \st{x}+s\st{h} \in K$ from our definition of the compact set $ K $, and hence $ \overrightarrow{[\st{y},\st{y}+\st{k}]}\subseteq U $, i.e. $ (\st{y},\st{k})\in \thick{U} $.
\end{proof}
\noindent Let us note that if $ x\in U^\bullet $ and $ h\approx 0 $, then $ (x,h) \in \thick{U}^\bullet $.

The following is the Fermat-Reyes theorem in the context of Colombeau generalized functions.
\begin{theo} 
\label{FRtheorem}
Let $U$ be open in $\R$ and let $f\in \gs(U)$. Then there exists a unique $r\in \gs(\thick{U})$
 such that
\begin{equation}\label{fr}
f(x+h) = f(x) + h\cdot r(x,h) \qquad \forall (x,h)\in \thick{U}^\bullet
\end{equation}
and $r(x,0) = f'(x)$.
\end{theo}
\begin{proof} To show existence, set $r_\eps(x,h) :=  \int_0^1 f_\eps'(x+sh)\,ds$ for every $ (x,h) \in \thick{U} $.
Then each $r_\eps$ is smooth, $r_\eps(x,0) = f_\eps'(x)$, $(r_\eps)$ is moderate,  and its class 
$r$ clearly satisfies (\ref{fr}). 

To show uniqueness, suppose that $\tilde{r}$ is another incremental ratio, and fix $ (x,h) \in \thick{U}^\bullet $. By Lemma \ref{nearStdThickIsOpen} we can find $ a>0 $ such that $ d_F(x,y)<a $ and $ d_F(h,k)< a $ imply $ (y,k) \in \thick{U}^\bullet $. On the open interval $ K:=(\st{h}-a,\st{h}+a) $ we can consider the generalized functions $ \rho_\eps(k) := r_\eps(x_\eps,k) $ and $ \tilde{\rho}_\eps(k) := \tilde{r}_\eps(x_\eps,k) $. Note that if $ k\in K $ then $ (x,k) \in \thick{U}^\bullet $; moreover, since $ (x_\eps) $ is near standard, both $ \rho_\eps$ and $\tilde{\rho}_\eps $ are moderate, so that $ \rho, \tilde{\rho}\in \gs(K) $ are well-defined. From (\ref{fr}) we get $ k \cdot \rho(k) = k \cdot \tilde{\rho}(k) $ for each $ k \in K^\bullet $. Therefore, $ \rho(k)=\tilde{\rho}(k) $ for each $ k\in K^\bullet $ invertible. Corollary \ref{invcor} implies that $ \rho=\tilde{\rho} $ and hence $ \rho(h)=r(x,h)=\tilde{\rho}(h)=\tilde{r}(x,h) $ for each $ (x,h) \in \thick{U}^\bullet $. Therefore, $ r=\tilde{r} $ once again by Corollary \ref{invcor}.
\end{proof}
\section{Conclusions}

The results we proved show that the sharp topology is not as far from the Euclidean one as it may seem at first glance. Indeed, we can say it is the natural generalization of the ordinary topology on $ \R $ including infinitesimal neighbourhoods. We have seen that this necessarily entails certain intuitively clear consequences, and that using the generalized absolute value we have instruments to work with the sharp topology which are formally very similar to the classical ones.

In the same line of ideas, the Fermat-Reyes theorem shows that the behaviour of Colombeau generalized functions can be considered formally sufficiently near to that of ordinary smooth functions.

One can pursue this point of view further and consider generalized functions as certain types of set-theoretical maps defined on, and with values in, subsets of generalized numbers. In particular, we are interested in maps obtained, like for $ \gs(\Omega) $ and $ \gs_\tau(\Omega) $, from nets $ (u_\eps) $ of ordinary smooth functions. In this case, the analogous Fermat-Reyes theorem becomes the key instrument to develop the differential calculus of this type of {\em generalized smooth maps}. This will be the subject of future works.


\end{document}